\documentclass[a4paper,12pt]{article}

\usepackage[ngerman, american]{babel} 
\usepackage[utf8]{inputenc}
\usepackage[a4paper, left=2.5cm, right=2.5cm, top=3cm, bottom=3cm]{geometry}
\usepackage{graphicx}
\usepackage{float}
\usepackage{color,psfrag}
\usepackage{amssymb}
\usepackage{amsmath} 
\usepackage{amsthm}
\usepackage{dsfont}
\usepackage{mathrsfs}
\usepackage[dvipsnames]{xcolor}
\usepackage[automark,headsepline]{scrpage2}
\usepackage{setspace}
\usepackage{comment}

\newcommand{\om}{\omega}
\newcommand{\Om}{\Omega}
\renewcommand{\phi}{\varphi}
\renewcommand{\rho}{\varrho}
\renewcommand{\epsilon}{\varepsilon}

\newcommand{\cD}{\mathcal{D}}
\newcommand{\cE}{\mathcal{E}}
\newcommand{\cF}{\mathcal{F}}

\newcommand{\R}{\mathds{R}}

\newcommand{\dx}{{\mathrm d}}
\newcommand{\Id}{\mathrm{Id}}
\newcommand{\I}{\mathds{1}}

\newcommand{\ew}{\mathbb{E}}
\newcommand{\re}{\mathrm{e}}
\newcommand{\ri}{\mathrm{i}}

\newtheorem{theorem}{Theorem}        
\newtheorem{lemma}[theorem]{Lemma}             
\newtheorem{corollary}[theorem]{Corollary}

\newtheorem{remark}[theorem]{Remark}
\newtheorem{proposition}[theorem]{Proposition}

\begin{document}

\title{Short-time behavior of solutions to Lévy-driven SDEs}
\author{Jana Reker\thanks{Ulm University, Institute of Mathematical Finance, 89081 Ulm, Germany, e-mail: jana.reker@uni-ulm.de}}
\maketitle
\begin{abstract}
We consider solutions of Lévy-driven stochastic differential equations of the form~$\dx X_t=\sigma(X_{t-})\dx L_t$, $X_0=x$ where the function $\sigma$ is twice continuously differentiable and maximal of linear growth and the driving L\'evy process $L=(L_t)_{t\geq0}$ is either vector or matrix-valued. While the almost sure short-time behavior of Lévy processes is well-known and can be characterized in terms of the characteristic triplet, there is no complete characterization of the behavior of the process $X$. Using methods from stochastic calculus, we derive limiting results for stochastic integrals of the from~$\smash{t^{-p}\int_{0+}^t\sigma(X_{t-})\dx L_t}$ to show that the behavior of the quantity~$t^{-p}(X_t-X_0)$ for $t\downarrow0$ almost surely mirrors the behavior of $t^{-p}L_t$. Generalizing $t^p$ to a suitable function $f:[0,\infty)\rightarrow\R$ then yields a tool to derive explicit LIL-type results for the solution from the behavior of the driving Lévy process.
\end{abstract}



\section{Introduction}
We consider the almost sure short-time behavior of the stochastic process $X=(X_t)_{t\geq0}$ which is the solution of a stochastic differential equation (SDE) of the form
\begin{equation}\label{eq-SDE}
\dx X_t=\sigma(X_{t-})\dx L_t,\ X_0=x\in\R^n
\end{equation}
driven by an $\R^d$-valued L\'evy process $L=(L_t)_{t\geq0}$. The function $\sigma:\R^{n}\rightarrow\R^{n\times d}$ is chosen to be twice continously differentiable and maximal of linear growth, which ensure that~\eqref{eq-SDE} has a unique strong solution (see e.g.~\cite[Thm.~7, p.~259]{Protter2005}) and that the It$\overline{\mathrm{o}}$ formula is applicable for the process $X$. The aim of this paper is to compare the behavior of $X$ at small times to that of suitable functions. For real-valued Lévy processes, results by Shtatland~\cite{Shtatland1965} and Rogozin~\cite{Rogozin1968} characterize the almost sure convergence of the quotient~$L_t/t$ for $t\downarrow0$ in terms of the total variation of the paths of the process, which was generalized to determining the behavior of the quotient for arbitrary positive powers of~$t$ in~\cite{BlumenthalGetoor1961},~\cite{Pruitt1981} and~\cite{BertoinDoneyMaller2008} from the characteristic triplet. The exact scaling function~$f$ for law of the iterated logarithm-type (LIL-type) results of the form $\limsup_{t\downarrow0}L_t/f(t)=c$ a.s. for a deterministic constant $c$ was determined by Khinchine for Lévy processes that include a Gaussian component (see e.g.~\cite[Prop.~47.11]{Sato2013}) and in e.g.~\cite{Savov2009} and \cite{Savov2010} for more general types of Lévy processes. The multivariate counterpart to these LIL-type results was derived in the recent paper~\cite{Einmahl2019}, showing that the short-time behavior of the driving process in~\eqref{eq-SDE} is already well-understood. For the solution $X$, the situation becomes more difficult. It was shown in~\cite{SchillingSchnurr2010} and~\cite{Kuehn2018} that $X$ is a so-called Lévy-type Feller process, i.e. the characteristic function of $X_t$ can be expressed using a characteristic triplet similar to the driving Lévy process with the triplet $(A(x),\nu(x),\gamma(x))$ additionally depending on the initial condition $x\in\R^n$ and the function $\sigma$. The short and long-time behavior of such Feller processes can be characterized in terms of power-law functions using a generalization of Blumenthal-Getoor indices (see~\cite{Schnurr2013}), where the symbol now plays the role of the characteristic exponent. Using similar methods, an explicit short-time LIL in one dimension was derived in~\cite{KnopovaSchilling2014}. The definition of a Lévy-type Feller process suggests that one can think of $X$ as ''locally Lévy'' and, since the short-time behavior of the process is determined by the path behavior in an arbitrarily small neighborhood of zero, the process $X$ thus should directly mirror the short-time behavior of the driving Lévy process. We confirm this hypothesis in terms of power-law functions in Proposition~\ref{prop-aslimitthroughSDE} and Theorem~\ref{prop-asdivergencethroughSDE} by showing that the almost sure finiteness of $\lim_{t\downarrow0}t^{-p}L_t$ implies the almost sure convergence of the quantity~$t^{-p}(X_t-X_0)$ and that similar results hold for $\limsup_{t\downarrow0} t^{-p}(X_t-X_0)$ and $\liminf_{t\downarrow0}t^{-p}(X_t-X_0)$ with probability one whenever $\lim_{t\downarrow0}t^{-p/2}L_t$ exists almost surely. Using knowledge on the form of the scaling function for the driving Lévy process, the limit theorems can be generalized to suitable functions~$f:[0,\infty)\rightarrow\R$ to derive explicit LIL-type results for the solution of~\eqref{eq-SDE}. As another application, we will also briefly study convergence in distribution and in probability, showing that results on the short-time behavior of the driving process translate here as well.

\section{Preliminaries}
A L\'evy process $L=(L_t)_{t\geq0}$ is a stochastic process with stationary and independent increments the path of which are almost surely c\`adl\`ag, i.e. right-continuous with finite left-limits, and start in 0 with probability one. In the following analysis, we consider both $\R^d$-valued and $\R^{n\times d}$-valued Lévy processes. By the usual convention, we identify~$\R^d$ with~$\R^{d\times 1}$, i.e. the elements are interpreted as column vectors. The transpose of a vector or matrix $x$ is denoted by $x^T$. Further, $\langle\cdot,\cdot\rangle$ and $\|\cdot\|$ denote the Euclidean scalar product and Euclidean norm on $\R^d$, respectively unless otherwise specified. By the L\'evy-Khintchine formula (see e.g.~\cite[Thm. 8.1]{Sato2013}), the characteristic function of an $\R^d$-valued L\'evy process~$L$ is given by
\begin{displaymath}
\phi_{L}(z)=\ew \re^{\ri z L_t} = \exp (t \psi_L(z)) , \ z\in \R^d,
\end{displaymath}
where $\psi_L$ denotes the characteristic exponent satisfying
\begin{align*}
\psi_L(z)=-\frac{1}{2}\langle z,A_Lz\rangle+i\langle \gamma_L,z\rangle+\int_{\R^d}\big(\exp(i\langle z,s\rangle)-1-i\langle z,s\rangle\I_{\{\|s\|\leq1\}}\big)\nu_L(\dx s), \ z\in \R^d.
\end{align*}
Here, $A_L\in\R^{d\times d}$ is the Gaussian covariance matrix, $\nu_L$ is the L\'evy measure and $\gamma_L\in\R^d$ is the location parameter of $L$. The characteristic triplet of $L$ is denoted by $(A_L,\nu_L,\gamma_L)$. See e.g. \cite{Sato2013} for any further information on L\'evy processes. Any $\R^{n\times d}$-valued L\'evy process can be seen as an $\R^{nd}$-valued L\'evy process by vectorization such that the above representations are valid in matrix case as well. Using the usual convention, a matrix $m$ is vectorized by writing its entries column-wise into a vector, which we denote by $m^{vec}$. Properties of $\nu_L$ in dimension one are sometimes also given in terms of its tail function, in which case we write $\smash{\overline{\Pi}^{(+)}_L}(x)=\nu_L((x,\infty))$, $\smash{\overline{\Pi}^{(-)}_L}(x)=\nu_L((-\infty,-x))$ and $\smash{\overline{\Pi}}_L(x)=\smash{\overline{\Pi}^{(+)}_L}(x)+\smash{\overline{\Pi}^{(-)}_L}(x)$ for~$x>0$. In higher dimensions, the interval $(x,\infty)$ is replaced by the set $\{y: \|y\|>x\}$. The abbreviation ''a.s.'' means ''almost sure(ly)''.

\medskip
L\'evy processes form an important subclass of semimartingales. For any c\`{a}dl\`{a}g process $X$, we denote by $X_{s-}$ the left-hand limit of $X$ at time~${s\in(0,\infty)}$ and by $\Delta X_s=X_s-X_{s-}$ its jumps. The process $X_{s-}$ is càglàd, i.e. left-continuous with finite right limits. Any integrals are interpreted as integrals with respect to semimartingales as e.g. in~\cite{Protter2005} and we generally consider a filtered probability space $\bigl(\Om,\cF,(\cF_t)_{t\geq0},P)$ satisfying the usual hypotheses (see e.g.~\cite[p.~3]{Protter2005}). The integral bounds are assumed to be included when the notation $\smash{\int_a^b}$ is used and the exclusion of the left or right bound is denoted by~$\smash{\int_{a+}^b}$ or $\smash{\int_a^{b-}}$. Let~$X$, $Y$ and $Z$ be semimartingales taking values in $\R^{n\times d}$, $\R^{d\times m}$ and $\R^{m\times d}$, respectively. Integrals with respect to matrix-valued semimartingales are interpreted as
\begin{align*}
\Big(\int_{(a,b]}X_{s-}\dx Y_s\Big)_{i,j}=\sum_{k=1}^d\int_{(a,b]}(X_{i,k})_{s-}\dx (Y_{k,j})_s,\\
\Big(\int_{(a,b]}\dx Z_sX_{s-}\Big)_{i,j}=\sum_{k=1}^d\int_{(a,b]}(X_{k,j})_{s-}\dx (Z_{i,k})_s.
\end{align*}
Note that the properties of one type of the multivariate stochastic integral readily carry over to the other by transposition of the matrix-valued semimartingales. The integration by parts formula takes the form
\begin{displaymath}
\int_{(0,t]}X_{s-}\dx Y_s=X_tY_t-X_0Y_0-\int_{(0,t]}\dx X_sY_{s-}-[X,Y]_{0+}^t
\end{displaymath}
in the matrix case. Let $L=(L_t)_{t\geq0}$ be an $\R^{d\times d}$-valued L\'evy process or semimartingale and $\Id\in\R^{d\times d}$ denote the identity matrix. Then the (strong) solution ${X=(X_t)_{t\geq0}}$ to the stochastic differential equation (SDE)
\begin{displaymath}
 \dx X_t=X_{t-}\dx L_t,\ t>0,\quad X_0=\Id
\end{displaymath}
is called (left) stochastic exponential of $L$ and denoted by $\smash{\overset{\leftarrow}{\cE}}(L)$ while the solution $Y$ to the~SDE
\begin{displaymath}
 \dx Y_t=\dx L_tY_{t-},\ t>0,\quad Y_0=\Id
\end{displaymath}
is called right stochastic exponential of $L$ and denoted by $\smash{\overset{\rightarrow}{\cE}}(L)$. As observed for the integrals, properties of the process $X$ carry over to $Y$ by transposition and vice versa. Unless specified otherwise, the term ''stochastic exponential'' refers to the left stochastic exponential and we omit the arrow. 

\section{Main Results}
The aim of this paper is a characterization of the a.s. short-time behavior of the solution to a Lévy-driven SDE by relating it to the behavior of the driving process. Since the short-time behavior of Lévy processes is already well-studied, we can use results from~\cite{BertoinDoneyMaller2008} and~\cite{Einmahl2019} to gain detailed insight in the behavior of the solution, as well as a method to derive LIL-type results for many frequently-used models. Note that the results given partially overlap with characterizations obtained from other approaches such as the generalization of Blumenthal-Getoor indices for Lévy-type Feller processes dicussed e.g. in~\cite{Schnurr2013} while also covering new cases such as a.s. limits for $t\downarrow0$. Whenever possible, we work with general semimartingales and include converse results to reobtain the limiting behavior of the driving process from the solution. As a first step, we give a lemma that characterizes the a.s. short-time behavior of a stochastic integral when the behavior of the integrand is known.

\begin{lemma}\label{lem-convergence-integral} 
Let $X=(X_t)_{t\geq 0}$ be a real-valued semimartingale, $p>0$ and $\phi=(\phi_t)_{t\geq0}$ an adapted c\`{a}gl\`{a}d process such that $\lim_{t\downarrow0}t^{-p}\phi_t$ exists and is finite with probability one. Then
\begin{equation}\label{lemma-limit} 
\frac{1}{t^p}\int_{0+}^t\phi_s\dx X_s \rightarrow 0\ \mathrm{a.s.\ for}\ t\downarrow0.
\end{equation}
\end{lemma}

\begin{proof} 
Define  the process $\psi$ ($\om$-wise) by
\begin{displaymath}
\psi_t:=\begin{cases}t^{-p}\phi_t,\ t>0,\\ \lim_{s\downarrow0} s^{-p}\phi_s,\ t=0,\end{cases}
\end{displaymath}
possibly setting $\psi_0(\om)=0$ on the null set where the limit does not exist. By definition,~$\psi$ is c\`{a}gl\`{a}d, and, as $\lim_{t\downarrow}t^{-p}\phi_t$ exists almost surely in~$\R$ and $\cF_0$ contains all null sets by assumption and the filtration is right-continuous,~$\psi_0$ is $\cF_0$-measurable. Therefore, $\psi$ is also adapted. This implies that the semimartingale
\begin{displaymath}
Y_t:=\int_{0+}^t\psi_s\dx X_s
\end{displaymath}
is indeed well-defined, allowing to rewrite the process considered in \eqref{lemma-limit} using the associativity of the stochastic integral. This leads to
\begin{displaymath}
\int_{0+}^t\phi_s\dx X_s=\int_{0+}^ts^p\psi_s\dx X_s=\int_{0+}^ts^p\dx Y_s,
\end{displaymath}
which implies
\begin{displaymath}
\frac{1}{t^p}\int_{0+}^t\phi_s\dx X_s=\frac{1}{t^p}\Bigl(t^pY_t-\int_{0+}^tY_s\dx (s^p)\Bigr)=Y_t-\frac{1}{t^p}\int_{0+}^tY_s\dx (s^p)
\end{displaymath}
by partial integration. As $Y$ is a semimartingale which has a.s. c\`{a}dl\`{a}g paths additionally satisfying $Y_0=0$ by definition, we have $\lim_{t\downarrow0}Y_t=0$ with probability one. The term remaining on the right-hand side is a path-by-path Lebesgue-Stieltjes integral. Note that, as $p>0$, the integrator is increasing, thus implying the monotonicity of the corresponding integral. This leads to
\begin{displaymath}
\inf_{0<s\leq t}Y_s\leq\frac{1}{t^p}\int_{0+}^tY_s\dx (s^p)\leq\sup_{0<s\leq t}Y_s.
\end{displaymath}
Recalling $\lim_{t\downarrow0}Y_t=0$ a.s., we can conclude that the above terms vanish with probability one as $t\downarrow0$, which yields the claim.
\end{proof}

\begin{remark}\rm 
(i) By defining the process $\psi_s=(\psi_{i,j})_{i,j}$ component-wise and considering~$Y_t$ as either $(Y_t)_{i,j}=\sum_{k=1}^d\int_{0+}^t(\psi_s)_{i,k}\dx (X_s)_{k,j}$ or $(Y_t)_{i,j}=\sum_{k=1}^d\int_{0+}^t(\psi_s)_{k,j}\dx (X_s)_{i,k}$, the lemma naturally extends to multivariate stochastic integrals with $\psi$ and $X$ being $\R^{n\times d}$-valued and $\R^{d\times m}$-valued semimartingales, respectively.\\
(ii) The function $t^p$ in the denominator may be replaced by an arbitrary continuous function $f:[0,\infty)\rightarrow\R$ that is increasing and satisfies $f(0)=0$ and $f(t)>0$ for all~$t>0$.
\end{remark}

Lemma~\ref{lem-convergence-integral} is the key tool to deriving a.s. short-time limiting results for the solution of a stochastic differential equation.

\begin{proposition}\label{prop-aslimitthroughSDE}
Let $L$ be an $\R^{d}$-valued semimartingale satisfying $L_0=0$, $v\in\R^{d}$, $p>0$, and $\sigma:\R^{d}\rightarrow\R^{n\times d}$ twice continuously differentiable and maximal of linear growth. Define $X=(X_t)_{t\geq0}$ as the solution of~\eqref{eq-SDE}. Then
\begin{displaymath}
\lim_{t\downarrow 0}\frac{L_t}{t^p}=v\ \mathrm{a.s.} \Rightarrow \lim_{t\downarrow 0}\frac{X_t-x}{t^p}=\sigma(X_0)v\ \mathrm{a.s.}
\end{displaymath}
\end{proposition}

\begin{proof}
Let $\lim_{t\downarrow0}t^{-p}L_t=v$ with probability one. By definition, $X$ satisfies the equation
\begin{displaymath}
X_t=x+\int_{0+}^t\sigma(X_{s-})\dx L_s.
\end{displaymath}
Applying partial integration to the individual components yields
\begin{align}
\Big(\frac{X_t-x}{t^p}\Big)_i&=\frac{1}{t^p}\sum_{k=0}^d\int_{0+}^t\sigma_{i,k}(X_{s-})\dx (L_k)_s\nonumber\\
&=\frac{1}{t^p}\sum_{k=0}^d\Big(\sigma_{i,k}(X_t)(L_k)_t-\sigma_{i,k}(x)(L_k)_0-\int_{0+}^t(L_k)_{s-}\dx\sigma_{i,k}(X_s)\nonumber\\
&\quad-\big[\sigma_{i,k}(X),L_k\big]_t\Big)\label{eq-partial-int}
\end{align}
As $t^{-p}L_t\rightarrow v$ a.s. by assumption and~$X_t\rightarrow x=X_0$ a.s. by definition of $X$, the first term on the right-hand side of~\eqref{eq-partial-int} converges almost surely to the desired limit as $t\downarrow0$. Thus, the claim follows if we can show that the remaining terms vanish when the limit is considered. Since $L_0=0$ a.s., this is true for the second term and, as $\sigma(X)$ is again a semimartingale, Lemma~\ref{lem-convergence-integral} is applicable for the third term of~\ref{eq-partial-int}, showing that it converges almost surely to zero. Since $\sigma$ is twice continuously differentiable, applying It$\overline{\mathrm{o}}$'s formula for $X$ in the quadratic covariation appearing in the last term yields
\begin{align}
\big[\sigma_{i,k}(X),L_k\big]_t&=\Big[\sigma_{i,k}(x)+\sum_{j=1}^n\int_{0+}^{\cdot}\frac{\partial\sigma_{i,k}}{\partial x_j}(X_{s-})\dx (X_j)_s\nonumber\\
&\quad+\frac{1}{2}\sum_{j_1,j_2=1}^n\int_{0+}^{\cdot}\frac{\partial^2\sigma_{i,k}}{\partial x_{j_1}\partial x_{j_2}}(X_{s-})\dx \big[X_{j_1},X_{j_2}\big]^c_s\nonumber\\
&\quad+ \sum_{0<s\leq \cdot} \Big(\sigma_{i,k}(X_s)-\sigma_{i,k}(X_{s-})-\sum_{j=1}^n\frac{\partial\sigma_{i,k}}{\partial x_j}(X_{s-})\Delta(X_j)_s\Big),L_k\Big]_t.\label{eq-appl-ito}
\end{align}
By linearity of the quadratic covariation, the right-hand side of~\eqref{eq-appl-ito} is split into seperate terms that can be treated individually. Further, using the associativity of the stochastic integral and the fact that continuous finite variation terms do not contribute to the quadratic covariation, it follows that many of the terms vanish, leaving
\begin{align}
\big[\sigma_{i,k}(X),L\big]_t&=\sum_{j=1}^n\int_{0+}^t\frac{\partial\sigma_{i,k}}{\partial x_j}(X_{s-})\dx [X_j,L_k]_s\label{eq-appl-ito2}\\
&\quad+\Big[\sum_{0<s\leq \cdot} \Big(\sigma_{i,k}(X_s)-\sigma_{i,k}(X_{s-})-\sum_{j=1}^n\frac{\partial\sigma_{i,k}}{\partial x_j}(X_{s-})\Delta(X_j)_s\Big),L_k\Big]_t.\nonumber
\end{align}
For the first term observe that the quadratic variation process is of finite variation, such that the integral is given by a path-by-path Lebesgue-Stieltjes integral. By the definition of $X$, it follows that
\begin{displaymath}
[X_j,L_k]_t=\sum_{l=1}^d\int_{0+}^t\sigma_{j,l}(X_{s-})\dx [L_l,L_k]_s.
\end{displaymath}
Denoting integration with respect to the total variation measure of a process $Y$ as $\dx TV_Y$, the individual integrals can be estimated by
\begin{align*}
&\Big|\int_{0+}^t\sigma_{j,l}(X_{s-})\dx [L_{l},L_{k}]_s\Big|\leq\int_{0+}^t\big|\sigma_{j,l}(X_{s-})|\dx TV_{[L_l,L_k]}(s)\\
&\leq\Bigl(\int_{0+}^t\big|\sigma_{j,l}(X_{s-})\big|\dx [L_l,L_l]_s\Bigr)^{\frac{1}{2}}\Bigl(\int_{0+}^t\big|\sigma_{j,l}(X_{s-})\big|\dx [L_k,L_k]_s\Bigr)^{\frac{1}{2}}\\
&\leq \sup_{0<s\leq t}\big|\sigma_{j,l}(X_{s-})\big|\sqrt{[L_l,L_l]_t}\sqrt{[L_k,L_k]_t},
\end{align*}
using the Kunita-Watanabe inequality (see e.g. \cite[Th.~II.25]{Protter2005}) and the fact that the resulting integrals have increasing integrators. Further, the above estimates also show that the total variation of $\int_{0+}^t\sigma_{j,l}(X_{s-})\dx[L_l,L_k]$ satisfies this estimate. For the quadratic variation terms note that since $(L_0)_{k,l}=0$ a.s. and
\begin{displaymath}
[L_k,L_k]_t=(L_k)_t^2-2\int_{0+}^t(L_k)_{s-}\dx (L_k)_s,
\end{displaymath}
it follows from the assumption and the one-dimensional version of Lemma~\ref{lem-convergence-integral} that
\begin{displaymath}
\lim_{t\downarrow0}\frac{1}{t^p}[L_k,L_k]_t=\lim_{t\downarrow0}\frac{1}{t^p}\sqrt{[L_l,L_l]_t}\sqrt{[L_k,L_k]_t}=0
\end{displaymath}
with probability one. Thus, 
\begin{align*}
0&\leq\lim_{t\downarrow0}\sup\Big|\frac{1}{t^p}\int_{0+}^t\sigma_{j,l}(X_{s-})\dx [L_l,L_k]_s\Big|=0\ \mathrm{a.s.},
\end{align*}
and a similar estimate holds true for the total variation of $\int_{0+}^t\sigma_{j,l}(X_{s-})\dx [L_l,L_k]_s$. Denoting the total variation process of $Y$ at $t$ by $TV(Y)_t$, we obtain the bound
\begin{displaymath}
\frac{1}{t^p}\Big|\sum_{j=1}^n\int_{0+}^t\frac{\partial\sigma_{i,k}}{\partial x_j}(X_{s-})\dx [X_j,L_k]_s\Big|\leq\sum_{j=1}^n\sup_{0<s\leq t}\Big|\frac{\partial\sigma_{i,k}}{\partial x_j}(X_{s-})\Big|\frac{1}{t^p}TV\Big([X_j,L_k]\Big)_t
\end{displaymath}
for the first term on the right-hand side of~\eqref{eq-appl-ito2}, showing that it vanishes almost surely when the limit $t\downarrow0$ is considered. Lastly, denote
\begin{displaymath}
\Big[\sum_{0<s\leq t} \Big(\sigma_{i,k}(X_s)-\sigma_{i,k}(X_{s-})-\sum_{j=1}^n\frac{\partial\sigma_{i,k}}{\partial x_j}(X_{s-})\Delta(X_j)_s\Big),L_k\Big]_t=:[J,L_k]_t
\end{displaymath}
for the jump term remaining in~\eqref{eq-appl-ito2}. Using the Kunita-Watanabe inequality and recalling that $[L_k,L_k]=o(t^p)$ by the previous estimate, it remains to consider the quadratic variation of the process $J$. Evaluating
\begin{align*}
[J,J]_t&=\sum_{0<s\leq t}(\Delta J_s)^2=\sum_{0<s\leq t} \Big(\sigma_{i,k}(X_s)-\sigma_{i,k}(X_{s-})-\sum_{j=1}^n\frac{\partial\sigma_{i,k}}{\partial x_j}(X_{s-})\Delta(X_j)_s\Big)^2,
\end{align*}
and noting that
\begin{displaymath}
\sup_{0<s\leq t}\Big|\frac{\partial^2\sigma_{i,k}}{\partial x_{j_1}\partial x_{j_2}}(X_{s-})\Big|<\infty
\end{displaymath}
for all $j_1,j_2=1,\dots,n$ and a fixed $t\geq0$ as $\sigma\in C^2$ and $X$ is a càdlàg process, we can conclude that
\begin{align*}
[J,J]_t\leq C \sum_{0<s\leq t}\|\Delta X_s\|^2=C \sum_{0<s\leq t}\|\sigma(X_{s-})\Delta L_s\|^2\leq C'\sum_{0<s\leq t}\|\Delta L_s\|^2\leq C'\sum_{k=1}^d[L_k,L_k]_t
\end{align*}
for some finite (random) constants $C,C'$. This shows that both terms in~\eqref{eq-appl-ito2} are indeed $o(t^p)$ and do not contribute when the limit $t\downarrow0$ in~\eqref{eq-partial-int} is considered. Hence, the limit is equal to~$\sigma(X_0)v$ almost surely, which is the claim.
\end{proof}

\begin{remark}\rm
(i) Observe that $\lim_{t\downarrow0}t^{-p}L_t=v$ implies $[L,L]=o(t^p)$ here. Whenever $L$ is a Lévy process, the same assumption yields $[L,L]_t=o(t^{2p})$ (see Lemma~\ref{lem-quadraticvariation} below).\\
(ii) Similar to Lemma~\ref{lem-convergence-integral}, one can replace $t^p$ by any other continuous, increasing function $f:[0,\infty)\rightarrow\R$ that satisfies $f(0)=0$ and $f(t)>0$ for all $t>0$.\\
(ii) Since the short-time behavior of the process is determined by its behavior in an arbitrarily small neighborhood of zero, Proposition~\ref{prop-aslimitthroughSDE} and many of the results below are also applicable when the solution of the SDE is only well-defined on some interval $[0,\varepsilon]$ with $\varepsilon>0$. Thus, the linear growth condition can be omitted if one replaces $t$ by $\min\{t,\varepsilon\}$ in the calculations.
\end{remark}

Whenever one can assure that $\sigma(X_{s-})$ is invertible, the implication in Proposition~\ref{prop-aslimitthroughSDE} is indeed an equivalence. This yields the following counterpart to~\cite[Thm.~4.4]{Schnurr2013} for almost sure limits at zero.
\begin{proposition}\label{prop-equivalence-exponential}
Let $L$ be an $\R^d$-valued semimartingale, $v\in\R^d$ and $p>0$ and $X$ the solution to~\eqref{eq-SDE}. Let further $\sigma:\R^d\rightarrow\R^{d\times d}$ be twice continuously differentiable, maximal of linear growth and such that $\sigma(X_{t-})$ has almost surely full rank for $t\geq0$, where we set $X_{0-}=x$. Then
\begin{equation}\label{theorem-equivalence} 
\lim_{t\downarrow0}\frac{L_t}{t^p}=v\ \mathrm{a.s.}\Leftrightarrow \lim_{t\downarrow0}\frac{X_t-x}{t^p}=\sigma(x)v\ \mathrm{a.s.}
\end{equation}
\end{proposition}

\begin{proof}
As Proposition~\ref{prop-aslimitthroughSDE} yields the first implication, let $\lim_{t\downarrow0}t^{-p}(X_t-x)=\sigma(X_0)v$ with probability one. Using that $\sigma(X_{s-})$ has almost surely full rank, we can recover $L$ from~$X$~via
\begin{displaymath}
L_t=\int_{0+}^t\bigl(\sigma(X_{s-})\bigr)^{-1}\dx X_s.
\end{displaymath}
Since $\lim_{t\downarrow0}X_t=x$ a.s., it is $\|\sigma(X_t)-\sigma(x)\|<1$ a.s. for sufficiently small $t>0$. This implies
\begin{align*}
\sigma(x)\sigma(X_t)^{-1}&=\bigl(\Id-(\Id-\sigma(X_t)\sigma(x)^{-1})\bigr)^{-1}=\sum_{k=0}^{\infty}\bigl(\Id-\sigma(X_t)\sigma(x)^{-1}\bigr)^k\\
&=\Id+\bigl(\Id-\sigma(X_t)\sigma(x)^{-1}\bigr)+R_t,
\end{align*}
where the Neumann series converges almost surely in norm. Observe that we have by Taylor's formula
\begin{displaymath}
\frac{1}{t^p}(\sigma(X_t)-\sigma(x))_{i,j}=\frac{1}{t^p}\sum_{k=1}^{n}\frac{\partial \sigma_{i,j}}{\partial x_k}(x)(X_t-x)_{i,j}+\frac{1}{t^p}r_{i,j}(t)
\end{displaymath}
where the remainder term satisfies $r_{i,j}(t)=O((X_t-x)^2)=o(t^p)$. Thus,
\begin{align*}
\lim_{t\downarrow0}\frac{1}{t^p}(\Id-\sigma(X_t)\sigma(x)^{-1}\bigr)&=\lim_{t\downarrow0}\frac{1}{t^p}(\sigma(x)-\sigma(X_t)\bigr)\sigma(x)^{-1}
\end{align*}
exists almost surely from which it follows that also $R_t=o(t^p)$ with probability one. Hence,
\begin{align*}
\sigma(x)\frac{1}{t^p}L_t&=\frac{1}{t^p}\int_{0+}^t\bigl(\sigma(x)(\sigma(X_{s-}))^{-1}-\Id\bigr)\dx X_s+\frac{1}{t^p}\int_{0+}^t\Id\ \dx X_s\\
&=\frac{1}{t^p}\int_{0+}^t\bigl(\sigma(x)(\sigma(X_{s-}))^{-1}-\Id\bigr)\dx X_s+t^{-p}\bigl(X_t-x\bigr).
\end{align*}
and we find that the limit for $t\downarrow0$ exists almost surely and is equal to $\sigma(X_0)v$ by Lemma~\ref{lem-convergence-integral} and the assumption. This yields the claim since $\sigma(X_0)$ has full rank with probability~one.
\end{proof}

Proposition~\ref{prop-equivalence-exponential} is in particular applicable for the stochastic exponential by vectorization of the matrix-valued stochastic processes. Here, the condition $\det(\Id+\Delta L_s)\neq0$ for all~$s\geq0$ ensures that the inverse $\cE(L)^{-1}$ is well defined (see~\cite{Karandikar1991}).
\begin{corollary}\label{cor-equivalence-exponential}
Let $L$ be an $\R^{d\times d}$-valued semimartingale satisfying $\det(\Id+\Delta L_s)\neq0$ for all~$s\geq0$, $v\in\R^{d\times d}$ and $p>0$. Then
\begin{equation}\label{theorem-equivalence} 
\lim_{t\downarrow0}\frac{L_t}{t^p}=v\ \mathrm{a.s.}\Leftrightarrow \lim_{t\downarrow0}\frac{\cE(L)-\Id}{t^p}=v\ \mathrm{a.s.}
\end{equation}
\end{corollary}

\begin{remark}\rm 
In the case that $L$ is a L\'evy process, the a.s. limit $v$ appearing for $p=1$ in Corollary~\ref{cor-equivalence-exponential} is the drift of $L$. A result by Shtatland and Rogozin (see \cite{Shtatland1965} and~\cite{Rogozin1968}) directly links the existence of this limit to the process having sample paths of bounded variation. Since the stochastic exponential $\cE(L)$ has paths of bounded variation iff this holds true for the paths of $L$, a similar connection can be made for $\cE(L)$. Denote by $BV$ the set of stochastic processes having sample paths of bounded variation, then Corollary~\ref{cor-equivalence-exponential} implies
\begin{displaymath}
\lim_{t\downarrow0}\frac{\cE(L)-\Id}{t}\ \text{exists}\ \mathrm{a.s.} \Leftrightarrow \lim_{t\downarrow0}\frac{L_t}{t}\ \text{exists}\ \mathrm{a.s.}\Leftrightarrow L\in BV\Leftrightarrow \cE(L)\in BV.
\end{displaymath}
\end{remark}

Considering Proposition~\ref{prop-aslimitthroughSDE} in the context of L\'evy processes yields the following result.
\begin{proposition}\label{prop-levydriven}
Let $L$ be an $\R^{d\times m}$-valued L\'evy process satisfying $\lim_{t\downarrow0}t^{-p}L_t=v$ a.s. for some $v\in\R^{d\times m}$, $p>0$  and let $X=(X_t)_{t\geq0}$ be an $\R^{n\times d}$-valued semimartingale. Then
\begin{displaymath}
\frac{1}{t^p}\int_{0+}^tX_{s-}\dx L_s\rightarrow X_0v\ \mathrm{a.s.},\ t\downarrow0.
\end{displaymath}
If additionally $\lim_{t\downarrow0}t^{-p}X_t=w$ for some matrix $w\in\R^{n\times d}$ such that $wv=0$, then
\begin{displaymath}
\frac{1}{t^{2p}}\int_{0+}^tX_{s-}\dx L_s\rightarrow 0\ \mathrm{a.s.},\ t\downarrow0.
\end{displaymath}
\end{proposition}

Note that the above proposition holds in particular when $X_t=\sigma(L_t)$ for a suitable function $\sigma$, but the dependence on the driving process is not needed to conclude the convergence. This is due to the following property of the quadratic variation of a Lévy process.

\begin{lemma}\label{lem-quadraticvariation}
Let $L$ be given as in Proposition~\ref{prop-levydriven}, then, almost surely, $[L,L]_t=o(t^{2p})$ as~$t\downarrow0$.
\end{lemma}

\begin{proof}
Using the Kunita-Watanabe inequality to estimate the individual components in the multivariate case, we may restrict the argument to $d=1$. Here, the quadratic variation~$[L,L]_t$ of $L$ is a L\'evy process of bounded variation given by
\begin{displaymath}
[L,L]_t=\sigma^2t+\sum_{0<s\leq t}(\Delta L_s)^2,
\end{displaymath}
with the constant $\sigma$ being the variance of the Gaussian part of $L$ (if present).
In the case that $p<1/2$, applying Khintchine's LIL (see e.g.~\cite[Prop.~47.11]{Sato2013}) implies that ${\lim_{t\downarrow0}t^{-p}L_t=0}$ a.s. holds for any L\'evy process. Since we also have $2p<1$, Shtatland's result yields, regardless of the value of $\sigma^2$, that
\begin{displaymath}
\frac{1}{t^{2p}}[L,L]_t=\frac{1}{t}[L,L]_t\cdot t^{1-2p}\rightarrow 0\ \mathrm{a.s.},\ t\downarrow0.
\end{displaymath}
In the case that $p=1/2$, Khinchine's LIL yields $\limsup_{t\downarrow0} L_t/\sqrt{t}=\infty$ a.s. if the Gaussian part of $L$ is nonzero. As the limit exists and is finite by assumption, the process~$L$ must satisfy $\sigma=0$. This implies that the quadratic variation process has no drift, so ${[L,L]_t=o(t)}$ a.s. by Shtatland's result. In the case that $p>1/2$, consider $L$ with its drift (if present) subtracted from the process. This neither changes the structure of the quadratic variation nor the assumption on the a.s. convergence, but ensures that~\cite[Thm.~2.1]{BertoinDoneyMaller2008} is applicable. Note that whenever $p>1$ and $L$ is of finite variation with non-zero drift, we have $\lim_{t\downarrow0}t^{-p}|L_t|=\infty$ by~\cite{Rogozin1968}, showing that this case is excluded by the assumption. It now follows from~\cite{BertoinDoneyMaller2008} that the L\'evy measure $\nu_L$ of~$L$ satisfies
\begin{displaymath}
\int_{[-1,1]}|x|^{1/p}\nu_L(\dx x)<\infty.
\end{displaymath}
Noting that $\Delta[L,L]_t=f(\Delta L_t)$ for $f(x)=x^2$, it follows that $\nu_{[L,L]}(B)=\nu_L(f^{-1}(B))$ for all sets $B\subseteq[-1,1]$. As we can now treat $\nu_{[L,L]}$ as an image measure, it is
\begin{displaymath}
\int_{[-1,1]}|x|^{1/2p}\nu_{[L,L]}(\dx x)=\int_{[0,1]}|x|^{1/2p}\nu_{[L,L]}(\dx x)=\int_{[-1,1]}|x|^{1/p}\nu_L(\dx x)<\infty.
\end{displaymath}
Thus, the quadratic variation satisfies the same integral condition with $2p$ instead of $p$. As $[L,L]_t$ is a bounded variation L\'evy process without drift, part (i) of~\cite[Thm.~2.1]{BertoinDoneyMaller2008} yields the claim in the last case.
\end{proof}

\begin{proof}[Proof of Proposition~\ref{prop-levydriven}]
First, let $X$ be a general semimartingale. Without loss of generality, we can assume $X_0=0$ a.s., since
\begin{displaymath}
\frac{1}{t^p}\int_{0+}^tX_{s-}\dx L_s=\frac{1}{t^p}\int_{0+}^t(X_{s-}-X_0)\dx L_s+\frac{1}{t^p}X_0L_t
\end{displaymath}
and $t^{-p}X_0L_t\rightarrow X_0v$ a.s. for $t\downarrow0$ by assumption. As $X$ is a semimartingale, we have
\begin{displaymath}
\frac{1}{t^p}\int_{0+}^tX_{s-}\dx L_s=\frac{1}{t^p}X_tL_t-\frac{1}{t^p}X_0L_0-\frac{1}{t^p}\int_{0+}^t\dx X_sL_{s-}-\frac{1}{t^p}[X,L]_{0+}^t
\end{displaymath}
by partial integration. Applying $X_0=0$ for the first two summands and Lemma~\ref{lem-convergence-integral} for the integral on the right-hand side, it follows that the terms vanish with probability one as~$t\downarrow0$. For the covariation, we have
\begin{displaymath}
\Big|\Bigl(\frac{1}{t^p}[X,L]_t\Bigr)_{i,j}\Big|\leq\sum_{k=1}^d\Big|\frac{1}{t^p}[X_{i,k},L_{k,j}]_t\Big|\leq\sum_{k=1}^d\frac{1}{t^p}\sqrt{[X_{i,k},X_{i,k}]_t}\sqrt{[L_{k,j},L_{k,j}]_t}
\end{displaymath}
by the Kunita-Watanabe inequality. As each component $L_{k,j}$ of $L$ is again a L\'evy process satisfying $\lim_{t\downarrow0}t^{-p}(L_{k,j})_t=v_{k,j}$ a.s., one can conclude that $[L_{k,j},L_{k,j}]_t=o(t^{2p})$ by Lemma~\ref{lem-quadraticvariation}. Therefore, it follows that $t^{-p}[X,L]_t\rightarrow0$ a.s. for $t\downarrow0$, yielding the first part of the proposition. Assume next that additionally $\lim_{t\downarrow0}t^{-p}X_t=w$ for some $w\in\R^{n\times d}$ with $wv=0$. One can argue similar to the proof of Lemma~\ref{lem-convergence-integral} and define an adapted stochastic process $\psi$ ($\om$-wise) by
\begin{displaymath}
\psi_t:=\begin{cases}t^{-p}X_t,\ t>0,\\ \lim_{s\downarrow0} s^{-p}X_s,\ t=0,\end{cases}
\end{displaymath}
possibly setting $\psi_0(\omega)=0$ on the null set where the limit does not exist. Using the associativity of the stochastic integral, rewrite
\begin{displaymath}
\int_{0+}^tX_{s-}\dx L_s=\int_{0+}^ts^p\psi_{s-}\dx L_s=\int_{0+}^ts^p\dx Y_s,
\end{displaymath}
where $Y_t=\int_{0+}^t\psi_{s-}\dx L_s$ is a well-defined semimartingale, and use integration by parts to obtain
\begin{equation}\label{eq-part-int}
\frac{1}{t^{2p}}\int_{0+}^tX_{s-}\dx L_s=\frac{1}{t^{2p}}\Bigl(t^pY_t-\int_{0+}^tY_s\dx (s^p)\Bigr)=\frac{1}{t^p}Y_t-\frac{1}{t^{2p}}\int_{0+}^tY_s\dx (s^p).
\end{equation}
By the first part of the proposition, it is
\begin{displaymath}
\lim_{t\downarrow0}\frac{1}{t^p}Y_t=\lim_{t\downarrow0}\frac{1}{t^p}\int_{0+}^t\psi_{s-}\dx L_s=\psi_0v=wv=0
\end{displaymath}
with probability one, while the path-wise Lebesgue-Stieltjes integral can be estimated by
\begin{displaymath}
\frac{1}{t^p}\inf_{0<s\leq t}(Y_s)_{i,j}\leq\frac{1}{t^{2p}}\int_{0+}^t(Y_s)_{i,j}\dx (s^p)\leq\frac{1}{t^p}\sup_{0<s\leq t}(Y_s)_{i,j},\ i=1,\dots,m,\ j=1,\dots n
\end{displaymath}
due to the integrand being an increasing function. As $wv=0$, both bounds converge to zero with probability one, hence
\begin{displaymath}
\lim_{t\downarrow0}\frac{1}{t^{2p}}\int_{0+}^tY_s\dx (s^p)=0
\end{displaymath}
almost surely. Thus, the limit for $t\downarrow0$ of~\eqref{eq-part-int} exists with probability one and is equal to zero.
\end{proof}

The above proposition is in particular applicable for solutions of Lévy-driven SDEs. An inspection of the proof of Proposition~\ref{prop-aslimitthroughSDE} shows that, almost surely,
\begin{displaymath}
[\sigma_{i,k}(X),L_k]_t=o([L_1,L_1]_t+\dots+[L_d,L_d]_t).
\end{displaymath}
Since $[L_k,L_k]_t=o(t^{2p})$ for any $k=1,\dots,d$ by Lemma~\ref{lem-quadraticvariation}, it follows that
\begin{displaymath}
[\sigma(X),L]_t=o(t^{2p})
\end{displaymath}
with probability one whenever $L$ is a Lévy process satisfying $\lim_{t\downarrow0}t^{-p}L_t=0$ a.s. and $X$ is the solution of~\eqref{eq-SDE}. We use this fact to consider the almost sure $\limsup$ and $\liminf$ behavior of the quotient $t^{-p}(X_t-x)$ including the divergent case. Note that the condition $\lim_{t\downarrow0}t^{-p/2}L_t=0$ a.s. is satisfied whenever $p/2>1/2$ and $\smash{\int_0^1x^{2/p}\nu_L(\dx x)<\infty}$ by~\cite[Thm.~2.1]{BertoinDoneyMaller2008}.

\begin{theorem}\label{prop-asdivergencethroughSDE}
Let $L$ be an $\R^{d}$-valued Lévy process such that $\lim_{t\downarrow0}t^{-p/2}L_t=0$ a.s. for some $p>0$. Further, let $\sigma:\R^{n}\rightarrow\R^{n\times d}$ be twice continuously differentiable and maximal of linear growth and define $X=(X_t)_{t\geq0}$ as the solution of the~SDE~\eqref{eq-SDE}. Then, almost surely,
\begin{equation}\label{eq-divergent1}
\lim_{t\downarrow0} \Big(\frac{X_t-x}{t^p}-\frac{\sigma(X_t)L_t}{t^p}\Big)=\lim_{t\downarrow0} \Big(\frac{X_t-x}{t^p}-\frac{\sigma(x)L_t}{t^p}\Big)=0.
\end{equation}
In particular, if $\sigma(x)$ has rank $d$, we have
\begin{equation}\label{eq-divergent2}
\lim_{t\downarrow0}\frac{\|L_t\|}{t^p}=\infty\ \mathrm{a.s.}\ \Rightarrow\ \lim_{t\downarrow0}\frac{\|X_t-x\|}{t^p}=\infty\ \mathrm{a.s.}
\end{equation}
\end{theorem}

\begin{proof}
Similar to the proof of Proposition~\ref{prop-aslimitthroughSDE}, we use integration by parts and rewrite
\begin{equation}\label{eq-rewrite-growth}
\frac{X_t-x}{t^p}-\frac{\sigma(X_t)L_t}{t^p}=-\frac{1}{t^p}\int_{0+}^t\dx\sigma(X_s)L_{s-}-\frac{1}{t^p}[\sigma(X),L]_t.
\end{equation}
The claim follows by showing that the desired limiting behavior for the right-hand side. For the term involving the quadratic variation, this is immediate from the previous calculations. Hence, it remains to study the behavior of the integral. Using the It$\mathrm{\overline{o}}$ formula once more yields
\begin{align*}
\Big(\int_{0+}^t\dx \sigma(X_s)L_{s-}\Big)_i&=\sum_{k=1}^d\int_{0+}^t(L_k)_{s-}\dx\sigma_{i,k}(X_s) \\
&=\sum_{k=1}^d\int_{0+}^t(L_k)_{s-}\dx\Big(\sigma_{i,k}(x)+\sum_{j=1}^n\int_{0+}^s\frac{\partial\sigma_{i,k}}{\partial x_j}(X_{r-})\dx (X_j)_r\nonumber\\
&\quad+\frac{1}{2}\sum_{j_1,j_2=1}^n\int_{0+}^s\frac{\partial^2\sigma_{i,k}}{\partial x_{j_1}\partial x_{j_2}}(X_{r-})\dx \big[X_{j_1},X_{j_2}\big]^c_r+(J_{i,k})_s\Big)\nonumber,
\end{align*}
where the jump term is again denoted by $J$ and the component of $\sigma$ included in it is carried as a subscipt. Observe that by associativity of the stochastic integral, it follows that
\begin{align*}
&\int_{0+}^t(L_k)_{s-}\dx\Big(\int_{0+}^s\frac{\partial^2\sigma_{i,k}}{\partial x_{j_1}\partial x_{j_2}}(X_{r-})\dx \big[X_{j_1},X_{j_2}\big]_r^c\Big)\\
&=\sum_{l_1=1}^n\sum_{l_2=1}^n\int_{0+}^t(L_k)_{s-}\frac{\partial^2\sigma_{i,k}}{\partial x_{j_1}\partial x_{j_2}}(X_{s-})\sigma_{j_1,l_1}(X_{s-})\sigma_{j_2,l_2}(X_{s-})\dx[L_{l_1},L_{l_2}]_s^c\\
&=:\sum_{l_1=1}^n\sum_{l_2=1}^n\int_{0+}^t(L_k)_{s-}M_{s-}\dx[L_{l_1},L_{l_2}]_s^c
\end{align*}
which is a sum of pathwise Lebesgue-Stieltjes integrals. Thus,
\begin{align*}
&\frac{1}{t^p}\Big|\int_{0+}^t(L_k)_{s-}\dx\Big(\int_{0+}^s\frac{\partial^2\sigma_{i,k}}{\partial x_{j_1}\partial x_{j_2}}(X_{r-})\dx \big[X_{j_1},X_{j_2}\big]_r^c\Big)\Big|\\
&\leq\frac{1}{t^p}\sum_{l_1=1}^d\sum_{l_2=1}^d\sup_{0<s\leq t}\big|(L_k)_{s-}M_{s-}\big|\sqrt{[L_{l_1},L_{l_1}]_t}\sqrt{[L_{l_2},L_{l_2}]_t}.
\end{align*}
As $\sigma$ is in particular $C^2$, the supremum on the right-hand side is bounded and we conclude that the bound obtained converges to zero with probability one by Lemma~\ref{lem-quadraticvariation}. For the jump term we have
\begin{displaymath}
\frac{1}{t^p}\Big|\int_{0+}^t(L_k)_{s-}\dx (J_{i,k})_s\Big|\leq\frac{1}{t^p}\sum_{0<s\leq t}|(L_k)_{s-}|\cdot|\Delta (J_{i,k})_s|
\end{displaymath}
by definition. However, since $\sigma\in C^2(\R^d)$, it follows from Taylor's formula that
\begin{displaymath}
|\Delta (J_{i,k})_s|=\Big|\sigma_{i,k}(X_s)-\sigma_{i,k}(X_{s-})-\sum_{j=1}^n\frac{\partial\sigma_{i,k}}{\partial x_j}(X_{s-})\Delta(X_j)_s\Big|\leq C\|\Delta X_s\|^2\leq C' \|\Delta L_s\|^2
\end{displaymath}
for some finite (random) constants $C,C'\geq0$ such that
\begin{displaymath}
\frac{1}{t^p}\Big|\int_{0+}^t(L_k)_{s-}\dx (J_{i,k})_s\Big|\leq \frac{1}{t^p}C'\sup_{0<s\leq t}|(L_k)_s|\sum_{j=1}^d[L_j,L_j]_t,
\end{displaymath}
which also converges a.s. to zero as $t\downarrow0$. For the last term, observe first that
\begin{displaymath}
\int_{0+}^t(L_k)_{s-}\dx\Big(\int_{0+}^s\frac{\partial\sigma_{i,k}}{\partial x_j}(X_{r-})\dx (X_j)_r\Big)=\sum_{l=1}^d\int_{0+}^t(L_k)_{s-}\frac{\partial\sigma_{i,k}}{\partial x_j}(X_{s-})\sigma_{j,l}(X_{s-})\dx (L_l)_s
\end{displaymath}
by the associativity of the stochastic integral. Including the summation over $k$ and $j$, this can be rewritten as
\begin{align*}
\sum_{k=1}^d\sum_{l=1}^d\int_{0+}^t(L_k)_{s-}\Big(\sum_{j=1}^d\frac{\partial\sigma_{i,k}}{\partial x_j}(X_{s-})\sigma_{j,l}(X_{s-})\Big)\dx (L_l)_s=\sum_{k=1}^d\sum_{l=1}^d\int_{0+}^t(L_k)_{s-}(M_{i,k,l})_{s-}\dx (L_l)_s,
\end{align*}
where we note that $\sup_{0<s\leq t}|(M_{i,k,l})_s|$ is bounded for any fixed small $t\geq0$ and continuous at $0$ due to the continuity of $\sigma$ and its derivatives. Since $\lim_{t\downarrow0}t^{-p/2}L_t=0$ with probability one, it follows that, almost surely, $\lim_{t\downarrow0}t^{-p/2}(L_k)_t(M_{i,k,l})_t$ exists. Thus, the second part of Proposition~\ref{prop-levydriven} is applicable and one can conclude that the integral also converges to zero with probability one.  Since $\lim_{t\downarrow0}t^{-p/2}L_t=0$ with probability one, we have
\begin{align}
0&\leq\Big\|\frac{\sigma(X_t)L_t}{t^p}-\frac{\sigma(x)L_t}{t^p}\Big\|\leq\Big\|\frac{\sigma(X_t)-\sigma(x)}{t^{p/2}}\Big\|\cdot\Big\|\frac{L_t}{t^{p/2}}\Big\|\nonumber\\
&\leq\sum_{j=1}^n\sup_{0<s\leq t}\Big\|\frac{\partial\sigma}{\partial x_j}(X_t)\Big\|\cdot\Big\|\frac{X_t-x}{t^{p/2}}\Big\|\cdot\Big\|\frac{L_t}{t^{p/2}}\Big\|.\label{eq-switchforsigmax}
\end{align}
As $t\downarrow0$, the first term converges with probablity one by the assumptions on $\sigma$ and Proposition~\ref{prop-aslimitthroughSDE} is applicable for the second one. Using that $\lim_{t\downarrow0}t^{-p/2}L_t=0$ almost surely, the right-hand side of~\eqref{eq-switchforsigmax} converges to zero with probability one as $t\downarrow0$. If $\sigma(x)$ has rank $d$, Equation~\eqref{eq-divergent2} follows immediately from the convergence result in~\eqref{eq-divergent1}.
\end{proof}

Theorem~\ref{prop-asdivergencethroughSDE} allows to characterize the a.s. short-time behavior of the solution to a Lévy-driven SDE in terms of power law functions. In order to derive precise LIL-type results, we now turn to more general functions. Note that, whenever the driving Lévy process has a Gaussian component, its a.s. short-time behavior is determined by Khintchine's~LIL~(see e.g.~\cite[Prop.~47.11]{Sato2013}). Hence, Lemma~\ref{lem-quadraticvariation} readily generalizes to continuous increasing functions~$f:[0,\infty)\rightarrow\R$ with $f(0)=0$ and $f(t)>0$ for all $t>0$, as any function $f$ such that~$\lim_{t\downarrow0}L_t/f(t)$ exists in $\R$ must satisfy $\sqrt{2t\ln(\ln(1/t))}/f(t)\rightarrow0$ and it follows
\begin{equation}\label{eq-limitqv}
\lim_{t\downarrow0}\frac{[L,L]_t}{(f(t))^2}=\lim_{t\downarrow0}\Big(\frac{[L,L]_t}{t}\frac{t}{2t\ln(\ln (1/t))}\frac{2t\ln(\ln (1/t))}{(f(t))^2}\Big)=0\ \mathrm{a.s.}
\end{equation}
by~\cite[Thm.~1]{Shtatland1965}. Thus, $[L,L]_t=o(f(t)^2)$ and we can replace the function $t^{p/2}$ for some~${p>0}$ in Theorem~\ref{prop-asdivergencethroughSDE} by $f$ in this case and obtain a precise short-time behavior for the solutions of stochastic differential equations that include a diffusion part. In the case that~$L$ does not include a Gaussian component, $[L,L]$ is a finite variation process without drift satisfying $\lim_{t\downarrow0}t^{-1}[L,L]_t=0$ a.s. by~\cite[Thm.~1]{Shtatland1965}, such that an argument similar to~\eqref{eq-limitqv} is still applicable if $f$ decays sufficiently fast as $t\downarrow0$. For the general case, we combine Theorem~\ref{prop-asdivergencethroughSDE} with the precise information on possible scaling functions derived in~\cite{Einmahl2019}. Note that the conditions of Corollary~\ref{cor-aslimsup-general} below immediately follow from Khinthchine's LIL whenever~$h\equiv1$, as the process does not include a Gaussian part by assumption.

\begin{corollary}\label{cor-aslimsup-general}
Let $L$ be a purely non-Gaussian $\R^{d}$-valued Lévy process and ${f:[0,\infty)\rightarrow\R}$ be of the form~$f(t)=\sqrt{t\ln(\ln(1/t)}h(1/t)^{-1}$, where $h:[0,\infty)\rightarrow[0,\infty)$ is a continuous and non-decreasing slowly varying function, such that the set of cluster points of $L_t/f(t)$ as $t\downarrow0$ is bounded with probability one. Further, let $\sigma:\R^{d}\rightarrow\R^{n\times d}$ be twice continuously differentiable and maximal of linear growth and define $X=(X_t)_{t\geq0}$ as the solution of~\eqref{eq-SDE}. Then, almost surely,
\begin{equation}\label{eq-divergent3}
\lim_{t\downarrow0} \Big(\frac{X_t-x}{f(t)}-\frac{\sigma(X_t)L_t}{f(t)}\Big)=\lim_{t\downarrow0} \Big(\frac{X_t-x}{f(t)}-\frac{\sigma(x)L_t}{f(t)}\Big)=0.
\end{equation}
In particular, if $\sigma(x)$ has rank $d$, we have
\begin{displaymath}
\lim_{t\downarrow0}\frac{\|L_t\|}{f(t)}=\infty\ \mathrm{a.s.}\ \Rightarrow\ \lim_{t\downarrow0}\frac{\|X_t-x\|}{f(t)}=\infty\ \mathrm{a.s.}
\end{displaymath}
\end{corollary}

\begin{proof}
As the scaling function is of the form $f(t)=t^{1/2}\ell(1/t)$ with a slowly varying function $\ell$ by assumption, the a.s. boundedness of the cluster points of $L_t/f(t)$ in~$\R^d$ implies that, for all $\varepsilon\in(0,1/2)$,
\begin{displaymath}
\lim_{t\downarrow0}\frac{L_t}{t^{(1/2-\varepsilon)}}=\lim_{t\downarrow0}\frac{L_t}{f(t)}\cdot \ell(1/t)t^{\varepsilon}=0
\end{displaymath}
with probability one. Thus, Theorem~\ref{prop-asdivergencethroughSDE} is applicable with $p/2=1/2-\varepsilon$, yielding
\begin{displaymath}
\lim_{t\downarrow0} \Big(\frac{X_t-x}{t^{1-2\varepsilon}}-\frac{\sigma(x)L_t}{t^{1-2\varepsilon}}\Big)=0
\end{displaymath}
almost surely. Using the explicit form of $f$ and choosing $\varepsilon\in(0,1/4)$, it follows that
\begin{displaymath}
\lim_{t\downarrow0}\Big(\frac{X_t-x-\sigma(x)L_t}{f(t)}\Big)=\lim_{t\downarrow0} \Big(\frac{X_t-x-\sigma(x)L_t}{t^{1-2\varepsilon}}\cdot\frac{t^{1-2\varepsilon}}{f(t)}\Big)=0
\end{displaymath}
with probability one, which is~\eqref{eq-divergent3}, and the remaining claims follow in analogy to the proof of Theorem~\ref{prop-asdivergencethroughSDE}.
\end{proof}

The above results show that the almost sure short-time LIL-type behavior of the driving Lévy process directly translates to the solution of the stochastic differential equation~\eqref{eq-SDE}. We also note the following statement for the conversion of the corresponding cluster set.
\begin{corollary}\label{cor-clustersets}
Under the assumptions of Corollary~\ref{cor-aslimsup-general} let $\limsup_{t\downarrow0}\|L_t\|/f(t)$ be bounded with probability one. Then there exists an a.s. cluster set ${A_X=C(\{X_t/f(t):t\downarrow0\})}$ for the solution $X$ of~\eqref{eq-SDE} which is obtained from the cluster set of ${A_L=C(\{L_t/f(t):t\downarrow0\})}$ of the driving Lévy process~$L$ via $A_X=\sigma(x)A_L$.
\end{corollary}
Corollary~\ref{cor-clustersets} implies in particular that $A_X$ shares the properties of $A_L$ derived in~\cite[Thm.~2.4]{Einmahl2019} and that there is a one-to-one correspondence between the cluster sets whenever~$\sigma(x)$ has rank $d$. As $\sigma(x)=\Id$ for the stochastic exponential, we have $A_X=A_L$ for this example, mirroring the statement of Corollary~\ref{cor-equivalence-exponential}.

\medskip
Lastly, we use Theorem~\ref{prop-asdivergencethroughSDE} to show that one can also translate more general limiting results at zero from the driving Lévy process to the solution of~\eqref{eq-SDE}. Here, convergence in distribution and convergence in probability are denoted by $\smash{\overset{\cD}{\rightarrow}}$ and $\smash{\overset{P}{\rightarrow}}$, respectively. As the short-time behavior of Brownian motion is well-known, $L$ is taken to be a purely non-Gaussian Lévy process and we further assume the drift of $L$, whenever existent, to be equal to zero. Note that sufficient conditions for the attraction of a Lévy process to normality are e.g. given in~\cite[Thm.~2.5]{DoneyMaller2002}. Thus, the conditions of Corollary~\ref{cor-Dconv} below are readily checked from the characteristic triplet of $L$ and e.g. satisfied for a Lévy process with a symmetric Lévy measure such as~$\nu_L(\dx x)=\exp(-|x|)\I_{[-1,1]}(x)\dx x$.

\begin{corollary}\label{cor-Dconv}
Let $d=1$, $L$ as specified above and assume that there is a continuous increasing function $f:[0,\infty)\rightarrow[0,\infty)$ such that $f(t)^{-1}L_t\smash{\overset{\cD}{\rightarrow}}Y$ as $t\downarrow0$, where the random variable $Y$ follows a non-degenerate stable law with index~$\alpha\in(0,2]$. Let further $\sigma:\R\rightarrow\R$ be twice continuously differentiable and maximal of linear growth and define $X=(X_t)_{t\geq0}$ as the solution of~\eqref{eq-SDE} such that the initial condition~$x\in\R^{n}$ satisfies $\sigma(x)\neq0$. Then
\begin{equation}\label{eq-Dconv}
\frac{X_t-x}{f(t)}\overset{\cD}{\rightarrow}\sigma(x)Y.
\end{equation}
Whenever $f$ is regularly varying with index $a\in(0,1/2]$ at zero,~\eqref{eq-Dconv} also holds if the random variable $Y$ is a.s. constant.
\end{corollary}

\begin{proof}
If $\alpha=2$, i.e. $Y$ is normally disributed, the convergence of $L_t/f(t)$ implies that
\begin{equation}\label{eq-MMconvergence}
\lim_{t\downarrow0}t\overline{\Pi}^{(\#)}_L(xf(t))=0
\end{equation}
for all $x>0$ and $\#\in\{+,-\}$ by~\cite[Prop.~4.1]{MallerMason2008}. Choosing $x=1$, note that the condition~\eqref{eq-MMconvergence} is not sufficient to imply the integrability of $\smash{\overline{\Pi}^{(\#)}_L(f(t))}$ over $[0,1]$. However, since the  distribution of $Y$ is non-degenerate, the scaling function  $f$ is regularly varying with index~$1/2$ at zero (see~\cite[Thm.~2.5]{DoneyMaller2002}) such that also
\begin{displaymath}
\lim_{t\downarrow0}t\overline{\Pi}^{\#}L(t^{1/2-\varepsilon})=0.
\end{displaymath}
This yields the estimate
\begin{equation}\label{eq-tailbound}
\overline{\Pi}^{\#}_L(t^{(1/2-\varepsilon)k})\leq\frac{C_t}{t^k}
\end{equation}
where $C_t$ is bounded as $t\downarrow0$ and the function is thus integrable over $[0,1]$ for $0\leq k<1$. By assumption, $L$ does not have a Gaussian component and the drift of the process is equal to zero whenever it is defined. Hence,
\begin{displaymath}
\int_{0+}^t\overline{\Pi}^{\#}(t^{(1/2-\varepsilon)k})\dx t<\infty
\end{displaymath}
for both $\#=+$ and $\#=-$ and thus $\lim_{t\downarrow0}t^{-(1/2-\varepsilon)k}L_t=0$ a.s. by~\cite[Thm.~2.1]{BertoinDoneyMaller2008}. Applying Theorem~\ref{prop-asdivergencethroughSDE}, we obtain
\begin{displaymath}
\lim_{t\downarrow0} \Big(\frac{X_t-x}{t^{k-2\varepsilon k}}-\frac{\sigma(x)L_t}{t^{k-2\varepsilon k}}\Big)=0
\end{displaymath}
with probability one. It now follows for $k-2\varepsilon k>1/2$ that
\begin{displaymath}
\lim_{t\downarrow0}\Big(\frac{X_t-x-\sigma(x)L_t}{f(t)}\Big)=\lim_{t\downarrow0}\Big(\frac{X_t-x-\sigma(x)L_t}{t^{k-2\varepsilon k}}\cdot\frac{t^{k-2\varepsilon k}}{f(t)}\Big)=0
\end{displaymath}
almost surely, which yields the desired convergence of $f(t)^{-1}(X_t-x)$. If $Y$ follows a nondegenerate stable law with index~$\alpha\in(0,2)$, the right-hand side of~\eqref{eq-MMconvergence} is to be replaced by the tail function~$\smash{\overline{\Pi}^{\#}_Y(x)}$~(see~\cite[Prop.~4.1]{MallerMason2008}) and it follows from the proof of~\cite[Thm.~2.3]{MallerMason2008} that the scaling function $f$ is regularly varying with index~$1/\alpha$ at zero in this case. Thus, we can derive a bound similar to~\eqref{eq-tailbound} and argue as before. Noting that \cite[Prop.~4.1]{MallerMason2008} does not require the law of the limiting random variable to be non-degenerate, the argument is also applicable if $Y$ is a.s. constant and $f$ is regularly varying with index~$a\in(0,1/2]$ at zero.
\end{proof}

One can also give a result for convergence in probability. The conditions can be checked directly from the characteristic triplet of the driving Lévy process using~\cite[Thm.~2.2]{DoneyMaller2002} and are e.g. satisfied for finite variation Lévy processes. As the limiting random variable is a.s. constant, the proof is immediate from Corollary~\ref{cor-Dconv}.

\begin{corollary}\label{cor-Pconv}
Let $d=1$, $L$ as above and assume that there is a continuous increasing function $f:[0,\infty)\rightarrow[0,\infty)$ that is regularly varying with index $a\in(0,1/2]$ at zero such that $f(t)^{-1}L_t\smash{\overset{P}{\rightarrow}}v$ for some finite value $v\in\R$ as $t\downarrow0$. Let further $\sigma:\R\rightarrow\R$ be twice continuously differentiable and maximal of linear growth and define $X=(X_t)_{t\geq0}$ as the solution of~\eqref{eq-SDE}. Then
\begin{displaymath}
\frac{X_t-x}{f(t)}\overset{P}{\rightarrow}\sigma(x)v.
\end{displaymath}
\end{corollary}

\section*{Acknowledgements}
I would like to thank Alexander Lindner for suggesting this topic and supervising my work on the project.
\renewcommand*{\bibname}{References}
\bibliographystyle{plain}
\bibliography{References}

\end{document}